\documentclass[10pt]{amsart}

\usepackage{relsize} 
\usepackage{booktabs}
\usepackage{hyperref}
\usepackage{color}
\usepackage{caption}
\usepackage{subcaption}

\usepackage{mathtools}
\usepackage{amsthm, amssymb}
\usepackage{amsmath}
\usepackage{tikz-cd}
\usepackage{tikz}
\usetikzlibrary{snakes, 
	3d, matrix,decorations.pathreplacing,calc,decorations.pathmorphing}
\usetikzlibrary{decorations.markings}
\usetikzlibrary{patterns}

\numberwithin{equation}{section}
\allowdisplaybreaks[1]

\makeatletter
\newcommand{\leqnomode}{\tagsleft@true\let\veqno\@@leqno}
\newcommand{\reqnomode}{\tagsleft@false\let\veqno\@@eqno}
\makeatother

\newcommand{\flag}{{\mathcal{F} \ell}}

\newcommand{\A}[2]{{A^{{(#1)}}_{{#2}}}}

\newcommand{\PP}[2]{{\varphi^{{(#1)}}_{{#2}}}}

\newcommand{\xii}[2]{{\xi^{{(#1)}}_{{#2}}}}

\newcommand{\diag}{{\text{diag}}}

\newcommand{\defi}[1]{{\textit{#1}}}

\newcommand{\C}{\mathbb{C}}
\newcommand{\Z}{\mathbb{Z}}

\newtheorem{theorem}{Theorem}[section]

\newtheorem{proposition}[theorem]{Proposition}
\newtheorem{corollary}[theorem]{Corollary}

\theoremstyle{definition}
\newtheorem{example}[theorem]{Example}
\newtheorem{definition}[theorem]{Definition}
\newtheorem{remark}[theorem]{Remark}

\begin{document}
%%%%%%%%%%%%%%%%%%%%%%%%%%%%%%%%%%%%%%%%%%%%%%%%%%%%%%%%%%%%%%%%%%%%%%%%% 	
\title[{Flag {B}ott manifolds of general {L}ie type}]{Flag {B}ott manifolds of general {L}ie type 
and their equivariant cohomology rings}
%\shorttitle{Flag {B}ott manifolds of general {L}ie type}

\author{Shizuo Kaji}
%\email{skaji@imi.kyushu-u.ac.jp}
\address{Institute of Mathematics for Industry, 
	Kyushu University, 
	Fukuoka 819-0395, 
	Japan}
\email{skaji@imi.kyushu-u.ac.jp}

\author{Shintar\^{o} Kuroki}
%\email{kuroki@xmath.ous.ac.jp}
\address{Faculty of Science, Department of Applied Mathematics, 
	Okayama	University of Science, 
	1-1 Ridai-cho Kita-ku Okayama-shi Okayama 700-0005, 
	JAPAN}
\email{kuroki@xmath.ous.ac.jp}

\author{Eunjeong Lee}
%\email{eunjeong.lee@ibs.re.kr}
\address{Center for Geometry and Physics, 
	Institute for Basic Science (IBS), 
	Pohang 37673, 
	Korea}
\email{eunjeong.lee@ibs.re.kr}

\author{Dong Youp Suh}
%\email{dysuh@kaist.ac.kr}
\address{Department of Mathematical Sciences,  
	KAIST, 
	291 Daehak-ro, Yuseong-gu, Daejeon 34141, 
	South Korea}
\email{dysuh@kaist.ac.kr}

\thanks{Kaji was partially supported by KAKENHI, Grant-in-Aid for Scientific Research (C) 18K03304.
	Kuroki was supported by JSPS KAKENHI Grant 	Number 17K14196.
Lee was supported by IBS-R003-D1.
	Suh was supported by Basic Science 
	Research Program through the National Research Foundation of Korea (NRF) 
	funded by the Ministry of  Science, ICT \& Future Planning 
	(No.~2016R1A2B4010823).
}
%\date{\today}

%%%%%%%%%%%%%%%%%%%%%%%%%%%%%%%%%%%%%%%%%%%%%%%%%%%%%%%%%%%%%%%%%%%%%%%%
\begin{abstract}
	In this article we introduce flag Bott manifolds of general Lie type as the total spaces of iterated flag bundles. 
	They generalize the notion of flag Bott manifolds and generalized Bott manifolds, and admit nice torus actions.
	We calculate the torus equivariant cohomology rings of flag Bott manifolds of general Lie type.
\end{abstract}
%%%%%%%%%%%%%%%%%%%%%%%%%%%%%%%%%%%%%%%%%%%%%%%%%%%%%%%%%%%%%%%%%%%%%%%%

\maketitle
%
%\blue{Question (Kuroki): 
%Somebody could you please write (and tell me) the following thing?
%Where do we need to use the property of $Z$, i.e., the centralizer of a circle subgroup? I think our result seems to generalize for more general iterated bundles whose fibres are $K/H$ by only changing $Z$ to $H$, where $H$ is a maximal rank compact connected subgroup of $K$. (though I think we need to take an appropriate coefficient when we compute the cohomology rings)}

%%%%%%%%%%%%%%%%%%%%%%%%%%%%%%%%%%%%%%%%%%%%%%%%%%%%%%%%%%%%%%%%%%%%%%%%
\section{Introduction}
A \defi{Bott tower} $\{M_j \mid 0 \leq j \leq m\}$ is a sequence of $\C P^1$-fibrations $\C P^1 \hookrightarrow M_j \to M_{j-1}$ such that $M_j$ is the induced projective bundle of the sum of two complex line bundles over $M_{j-1}$. Each manifold $M_j$ is called \defi{a $j$-stage Bott manifold} and it is known that $M_j$ is a non-singular projective toric variety. 
On the other hand, a flag manifold is the orbit space $G/P$ of a complex Lie group $G$ divided by a parabolic subgroup $P$.
A flag manifold is known to be a non-singular projective variety having a nice torus action.

These two families of spaces are closely related by the Bott-Samelson resolution (see~\cite{Dem, GrKa94}).
Both families have been actively studied as spaces with nice torus actions, and have served as a test ground for various theories and problems. \emph{Schubert calculus} studies the cohomology of flag manifolds, in which topology, algebraic geometry, combinatorics, and representation theory meet together (see \cite{Ku} for a survey).
On the other hand, \emph{the cohomological rigidity problem of quasi-toric manifolds} may be regarded as one of the essential problems in toric topology,
and
some affirmative results are known for Bott manifolds (see~\cite{ Choi15,ChMa12, CMM15, Is12}).

There are two known natural generalizations of Bott manifolds: flag Bott manifolds, and generalized Bott manifolds. The \defi{flag Bott manifolds} extend the relation between Bott manifolds and Bott--Samelson manifolds. We refer the readers to~\cite{KLSS18} for the definition of flag Bott manifolds, and to~\cite{FLS18} for this enlarged relation. And the \defi{generalized Bott manifolds} are toric manifolds studied in~\cite{CMS-Trnasaction, CPS12, MaSu08}.

In this note, we introduce \defi{flag Bott manifolds of general Lie type} which simultaneously generalize 
flag manifolds, flag Bott manifolds, and generalized Bott manifolds.
We give two closely related descriptions of the flag Bott manifold of general Lie type; 
as the total space of an iterated flag bundle (Definition \ref{def_fBT_pullback})
 and as an orbit space of a Lie group (Definition \ref{def_fBT_quotient}).
We see a torus acts on the flag Bott manifold in a natural manner (Definition~\ref{def:canonical_action_of_T}). Moreover, we determine its Borel equivariant cohomology ring.
Theorem \ref{thm_equiv_cohomology_ring_of_FBM} unifies the known formulae for the equivariant cohomology 
of flag manifolds, flag Bott manifolds, and generalized Bott manifolds.

%%%%%%%%%%%%%%%%%%%%%%%%%%%%%%%%%%%%%%%%%%%%%%%%%%%%%%%%%%%%%%%%%%%%%%%
\section{Flag bundle and its cohomology}
%%%%%%%%%%%%%%%%%%%%%%%%%%%%%%%%%%%%%%%%%%%%%%%%%%%%%%%%%%%%%%%%%%%%%%%

Let $K$ be a compact connected Lie group and $T \subset K$ be a maximal torus of $K$. 
Let $Z \subset K$ be the centralizer in $K$ of a circle subgroup of $T$. 
Then, $Z$ is known to be connected. 
	Indeed, for any element $g\in Z$ consider the subgroup $H$ which is generated by 
	$g$ and the circle subgroup. Since $H$ is abelian, it is contained in some torus
	which is contained in $Z$. Therefore, $g$ is in the identity component of $Z$.
We denote by $W$ (resp. $W(Z)$) the Weyl group of $K$ (resp. $Z$).
The space $K/Z$ of left-cosets is called the \textit{generalized flag manifold},
and there exists the universal flag bundle $K/Z \hookrightarrow BZ \to BK$. For any map $f \colon X \to BK$ from a topological space $X$, we have the pull-back bundle $K/Z \hookrightarrow F_f(Z) \to X$, which fits in the diagram
\begin{equation}\label{eq_associated_flag_bundle}
\begin{tikzcd}
K/Z \arrow[r,equal] \arrow[d, hook] & K/Z \arrow[d, hook] \\
F_f(Z) \dar \arrow[r, "\tilde{f}"] & BZ \dar  \\
X \arrow[r, "f"]& BK
\end{tikzcd}
\end{equation}
The pull-back bundle is called the \defi{flag bundle over $X$ associated to the classifying map $f$ with fiber $K/Z$}.

\begin{example}\label{ex:flag}
Let $K=U(n)$ and $f: X\to BU(n)$.
We denote by $U(1)^{n}$ the set of all diagonal matrices in $K$.
It is well-known that $U(1)^{n}$ is a maximal torus in $K$.
We exhibit three examples of the associated flag bundles over~$X$.

Take the circle subgroup $\{\textrm{diag}(t,\dots,t,t^2) \mid t \in S^1\} \subset U(1)^{n}$.
Then its centralizer $Z_{1}$ is the group of the block diagonal matrices $U(n-1)\times U(1)\subset K$.
The associated flag bundle
$F_f(Z_{1})$ is isomorphic to
the projective bundle $\mathbb{P}(E)$
associated to the complex vector bundle $E$ classified by $f$.

We next take the circle subgroup $\{\textrm{diag}(t,t^2,\dots,t^n) \mid t \in S^1\} \subset U(1)^{n}$.
Then its centralizer $Z_{2}$ coincides with the group of the diagonal matrices~$U(1)^{n}$.
The associated flag bundle
$F_f(Z_{2})$ is isomorphic to the full flag bundle 
\[
\flag(E)=\{ V_0\subset V_1 \subset V_2 \subset \cdots \subset V_n=E \mid \mathrm{rank}(V_i)=i \}
\]
associated to $E$.

More generally, for positive integers $p_1,\dots,p_k > 0$ with $p_1+\cdots+p_k = n$, the centralizer $Z_{3}$ of the circle subgroup
	\[
	\{\textrm{diag}(\underbrace{t,\dots,t}_{p_1},\underbrace{t^2,\dots,t^2}_{p_2},\dots,\underbrace{t^k,\dots,t^k}_{p_k})\mid t \in S^1\} \subset U(1)^{n}
	\]
	is the group of the block diagonal matrices $U(p_1) \times \cdots \times U(p_k) \subset U(n)$. 
The associated flag bundle
$F_f(Z_{3})$ is isomorphic to the partial flag bundle
	\[
	\{V_0 \subset V_1 \subset V_2 \subset \cdots \subset V_k \subset E \mid \mathrm{rank}(V_i) = p_1+\cdots+p_i \}
	\]
	associated to $E$.
\end{example}

The cohomology of $F_f(Z)$ can easily be computed for some coefficient rings.
A prime $p$ is said to be a torsion prime of $K$ if $H_{\ast}(K;\mathbb{Z})$ has $p$-torsion. 
Let $R$ be a PID in which torsion primes of $K$ are invertible.
For example, if $K$ is $U(n)$ or $Sp(n)$, then we may take $R$ to be $\mathbb{Z}$.
When $K$ is simply-connected, the torsion primes are summarized in Table~\ref{table_torsion_prime} (see \cite[\S 2.5]{borel1961sous}).
Note that if $p$ is not a torsion prime of $K$, it is not of any circle centralizer $Z$ as well.

\begin{table}[ht]
	\begin{tabular}{|c|ccccccccc|}
		\hline
		Lie type & $A$ & $B$ & $C$ & $D$ & $G_2$ & $F_4$ & $E_6$ & $E_7$ & $E_8$\\
		\hline
		torsion primes & $\emptyset$ & $2$ & $\emptyset$ & $2$ & $2$ & $2,3$ & $2,3$ & $2,3$ & $2,3,5$ \\
		\hline
	\end{tabular}
	\caption{Torsion primes.}
	\label{table_torsion_prime}
\end{table}

\begin{proposition}\label{prop_cohomology_ring_of_associated_flag_bundle}
	Let $R$ be a PID in which the torsion primes of $K$ are invertible. 
	Then, we have a ring isomorphism
	\[
	H^{\ast}(F_f(Z);R) \cong H^{\ast}(X;R) \otimes_{H^{\ast}(BK;R)} H^{\ast}(BZ;R),
	\]
	where $H^{\ast}(X;R)$ has the $H^{\ast}(BK;R)$-module structure induced by the map~$f$ and $H^*(BZ;R)$ also has the natural $H^*(BK;R)$-module structure induced by the classifying map of the inclusion $Z\to K$.
\end{proposition}
\begin{proof}    
	Taking the cohomology of the lower square of \eqref{eq_associated_flag_bundle} gives rise to a homomorphism
	\[
	\varphi: H^{\ast}(X;R) \otimes_{H^{\ast}(BK;R)} H^{\ast}(BZ;R)\to H^{\ast}(F_f(Z);R),
	\]
	which we will show is an isomorphism.
	There are elements $\alpha_i\in H^{\ast}(BZ;R)$ which restrict to a basis of $H^{\ast}(K/Z;R)$ by \cite[\S 4.2]{borel1961sous}.
	By the Leray--Hirsch theorem we see 
	$H^{\ast}(BZ;R)$ is generated by $\{\alpha_i\}$ over $H^{\ast}(BK;R)$ and 
	$H^{\ast}(F_f(Z);R)$ is generated by $\{\tilde{f}^*\alpha_i\}$ over $H^{\ast}(X;R)$.
	Since $\varphi$ is a ring homomorphism which is $H^{\ast}(X;R)$-module isomorphism, it is a ring isomorphism.
\end{proof}

\begin{example}
Let $K = U(3)$ and $W\cong \mathfrak{S}_3\cong \langle s_1, s_2 \rangle$.
Here we may regard $s_{1}$ as the permutation matrix of the 1st and the 2nd coordinates and $s_{2}$ as that of the 2nd and the 3rd coordinates.
Suppose that $Z$ is the centralizer subgroup of the circle subgroup $\{\textrm{diag}(t,1,1) \mid t \in S^1\} \subset U(1)^{3}$. 
Note that the Lie algebra of this circle subgroup is fixed by the permutation $s_{2}$, and 
$Z$ is isomorphic to the block diagonal matrices $U(1)\times U(2)$,
thus, we have $K/Z \cong \C P^2$, and 
\[
H^*(BK;\Z) \cong \Z[x_1,x_2,x_3]^{\mathfrak{S}_3}, \quad
H^*(BZ; \Z) \cong \Z[x_1,x_2,x_3]^{W(Z)}
\]
where $W(Z) = \langle s_2 \rangle \subset \mathfrak{S}_3$. 
Let $f: X\to BU(3)$ be a map and 
$E \to X$ be the complex vector bundle of rank $3$ classified by $f$.
The induced projective bundle $\mathbb{P}(E)$ coincides with~$F_f(Z)$.
Hence, by applying Proposition~\ref{prop_cohomology_ring_of_associated_flag_bundle}, 
we can compute its cohomology as follows:
\begin{align*}
H^{\ast}(\mathbb{P}(E);\Z) &\cong H^{\ast}(X;\Z) \otimes_{H^{\ast}(BK;\Z)} H^{\ast} (BZ;\Z) \\
&\cong H^{\ast}(X;\Z) \otimes_{\Z[x_1,x_2,x_3]^{\mathfrak{S}_3}} \Z[x_1,x_2,x_3]^{\langle s_2\rangle} \\
& \cong (H^{\ast}(X;\Z) \otimes \Z[x_1, x_2+x_3, x_2x_3])/ I, 
\end{align*}
where $I = \langle x_1+x_2+x_3 -c_1(E), x_1x_2 + x_2x_3 + x_3x_1- c_2(E), x_1x_2x_3- c_3(E)\rangle $.
Observe the following relations given by the ideal $I$:
\[
\begin{split}
x_2 + x_3 &= c_1(E) - x_1, \\
x_2x_3 &= c_2(E) - x_1(x_2+x_3) = c_2(E) - x_1(c_1(E) - x_1) = c_2(E) - x_1c_1(E) + x_1^2, \\
x_1x_2x_3 &= c_3(E).
\end{split}
\]
Denoting $x_1$ by $x$ and eliminating $x_{2}$ and $x_{3}$, we get the following formula:
\[
H^{\ast}(\mathbb{P}(E);\Z) \cong H^{\ast}(X;\Z)[x]/\langle x^3 - x^2 c_1(E) + xc_2(E) - c_3(E) \rangle.
\]
Similarly, for any complex vector bundle $E \to X$ of rank $n+1$, we have that:
\begin{equation}\label{eq:BH}
H^{\ast}(\mathbb{P}(E);\Z ) \cong H^{\ast}(X;\Z) [x] \bigg/ \left\langle \sum_{k=1}^{n+1} (-1)^k x^{n+1-k}c_{k}(E) \right\rangle,
\end{equation}
which recovers the well-known Borel--Hirzebruch formula (see~\cite[Chapter V, \S 15]{BorelHirzebruch58}).
\end{example}

As a corollary of Proposition ~\ref{prop_cohomology_ring_of_associated_flag_bundle}, 
we obtain a quick proof of~\cite[Proposition 21.17 and Remarks 21.18, 21.19]{BoTu82}.
\begin{corollary}
	Let $\flag(E) \to X$ be the full flag bundle associated to an $(n+1)$-dimensional complex vector bundle. 
	Then,  we have that
	\[
	H^{\ast}(\flag(E); \mathbb{Z}) \cong H^{\ast}(X;\mathbb{Z})[x_1,\dots,x_{n+1}] \bigg/
	\left\langle \prod_{k=1}^{n+1} (1+x_k) - c(E) \right\rangle.
	\]
\end{corollary}
\begin{proof}
	The flag bundle $\flag(E) \to X$ with fiber $\flag(n+1) := \flag(\C^{n+1})$ fits into the following diagram:
	\[
	\begin{tikzcd}
	\flag(n+1) \arrow[r, equal] \arrow[d, hook]& \flag(n+1) \arrow[d, hook]\\
	\flag(E) \dar \rar& BT \dar \\
	X \arrow[r, "f"]& BU(n+1)
	\end{tikzcd}
	\]
	where $f$ is the classifying map of the vector bundle $E \to X$ and $T$ is a maximal torus in $U(n+1)$. Applying Proposition~\ref{prop_cohomology_ring_of_associated_flag_bundle}, we have that
	\[
	H^{\ast}(\flag(E); \mathbb{Z}) \cong H^{\ast}(X; \mathbb{Z}) \otimes_{H^{\ast}(BU(n+1);\Z)} H^{\ast}(BT; \mathbb{Z}).
	\]
We identify $H^{\ast}(BT;\mathbb{Z})\cong \Z[x_1,x_2,\ldots,x_{n+1}]$ and $H^{\ast}(BU(n+1);\Z) \cong \Z[c_1,c_2,\ldots,c_{n+1}]$, where $c_i$ is the $i$th elementary symmetric polynomial in $x_{1}, \ldots,x_{n+1}$ for any~$i$. 
The assertion follows from the fact $f^*(c_i)=c_i(E)$.
\end{proof}

%\begin{remark}\blue{(the meaning of this remark is not very clear in the context. shall we remove?)}
%	It is known that flag manifolds of type $A$ and $C$ are realized as iterated projective bundles (see~\cite[\S 2]{KurokiSuh15cohomological} for more details on type $C$), and hence, the Borel--Hirzebruch formula \eqref{eq:BH}
%	can be applied to compute their cohomology rings. 
%\end{remark}

%%%%%%%
\section{Flag Bott manifold of general Lie type}
\label{sec_flag_Bott_manifolds}

In this section, we introduce the main object of our study, \emph{flag Bott towers of general Lie type}.
We give two closely related definitions of a flag Bott tower of general Lie type, and prove that they are equivalent when relevant groups are simply-connected.
For $1 \leq j \leq m$, let $K_j$ be a compact connected Lie group, $T_j \subset K_j$ be a maximal torus, and $Z_j \subset K_j$ be the centralizer of a circle subgroup of $T_j$. 

\begin{definition}\label{def_fBT_pullback}
	An \defi{$m$-stage flag Bott tower $F_{\bullet} = \{F_j \mid 0 \leq j \leq m\}$ of general Lie type \textup{(}\textup{or} an $m$-stage flag Bott tower\textup{)} associated to $(K_{\bullet}, Z_{\bullet}) = \{(K_j, Z_j) \mid 1 \leq j \leq m\}$}  is defined recursively as follows:
	\begin{enumerate}
		\item $F_0$ is a point.
		\item $F_j$ is the flag bundle over $F_{j-1}$ with fiber $K_j/Z_j$ associated to a map 
		\[
		f_j \colon F_{j-1} \to B K_j,
		\]
		where $f_j$ factors through $BT_j$.
	\end{enumerate}
\end{definition}
The requirement for $f_j$ to factor through $BT_j$ means that we consider those bundles which are the sum of line bundles.

Two flag Bott towers $F_{\bullet}$ and $F_{\bullet}'$ are \defi{isomorphic} if there is a collection of diffeomorphisms 
$F_j \to F_j'$ which commute with the projections $p_j \colon F_j \to F_{j-1}$ and $p_j' \colon F_j' \to F_{j-1}'$ for all $1 \leq j \leq m$.

Note that the fiber of each stage is a flag manifold, which admits a cell decomposition involving only even dimensional cells.
Moreover, the total space $E$ of a fiber bundle $F \hookrightarrow E \to B$ has the structure of a CW-complex whose cells are
	the product of those in the base $B$ and the fiber $F$ (see, for instance, \cite[p.105]{Novikov}). 
More precisely, for the fiber bundle $F \hookrightarrow E \to B$, the total space $E$ can be decomposed into the pull backs of all cells of $B$. The pull-back of a cell $c \subset B$ is homeomorphic to the product $F\times c$ (because $c$ is contractible), and $F\times c$ is also decomposed into the product of cells of $F$ and $c$.
Because the product of two disks is homeomorphic to the disk, this gives a cell decomposition of the total space $E$.
Therefore, a flag Bott manifold $F_{j}$ admits a cell decomposition involving only even dimensional cells as well, and in particular,
it is simply-connected.

\begin{example}\label{example_fBT}
Let $(K_j,Z_j)=(U(n_j+1),T^{n_j+1})$, where $T^{n_{j}+1}$ is a maximal torus of $U(n_j+1)$.
We have $K_j/Z_j \cong \flag(n_j+1)$ for each $j$, and we get an \defi{$m$-stage flag Bott tower} which is introduced in \cite[Definition 2.1]{KLSS18}. An $m$-stage flag Bott tower $F_{\bullet} = \{F_j \mid 0 \leq j \leq m\}$ is defined to be
an iterated bundle
	\[
	F_m \to F_{m-1} \to \cdots \to F_1 \to F_0 = \{\textup{a point}\},
	\]
	of manifolds $F_j = \flag\left(\bigoplus_{k=1}^{n_j+1} \xii{j}{k} \right)$
	where $\xii{j}{k}$ is a complex 
	line bundle over $F_{j-1}$ for each $1 \leq k \leq n_j+1$
	and $1\leq j \leq m$.
Since the flag Bott tower in Definition~\ref{def_fBT_pullback} is the generalization of flag Bott tower in \cite{KLSS18},
we call the latter a \defi{full flag Bott tower of type A} in this paper whenever we need to specify them.
\end{example}
\begin{example}
Let $(K_j,Z_j)=(U(n_j+1),U(n_j)\times U(1))$.
	We have {$K_j/Z_j \cong \C P^{n_j}$} for each $j$, and 
	we get an \defi{$m$-stage generalized Bott tower} 
	which is defined in~\cite{ CMS10-quasitoric, CMS-Trnasaction} to be
	an iterated bundle
	\[
	B_m \to B_{m-1} \to \cdots \to B_1 \to B_0 = \{\textup{a point}\},
	\]
	of manifolds $B_j = \mathbb{P}\left(\bigoplus_{k=1}^{n_j+1} \xii{j}{k} \right)$
	where $\xii{j}{k}$ is a complex 
	line bundle over $B_{j-1}$ for each $1 \leq k \leq n_j+1$
	and $1\leq j \leq m$.
\end{example}

\begin{remark}
		Iterated flag bundles have been studied well in the literature.
		For instance, in~\cite{He}, the author studies the cohomology rings of 
		iterated flag bundles associated to vector bundles
		which do not necessarily split into line bundles. He called them 
		\textit{Bott tower of flag manifolds}.
		A Bott tower of flag manifolds admits a torus action if the base space 
		admits one and the vector bundle is equivariant.
		However, the essential point of our construction is to restrict to 
		vector bundles which are the sum of line bundles.
		In this case, we obtain two views of our flag Bott manifold of general 
		Lie type (in this section \S\ref{sec_flag_Bott_manifolds}) and
		can introduce the bigger torus action (\S\ref{sec_equi_cohomology_ring_flag_Bott_manifolds}), which are non-trivial.

%There are plenty of many studies about flag bundles in literature. For instance, in \cite{He}, the author studies the cohomology rings of associated flag bundles, and he also considers iterated flag bundles which are associated from vector bundles (which may not be split into line bundles) called \textit{Bott tower of flag manifolds} in the previous version of~\cite{He}. 
%Note that it is different from the full flag Bott towers of type~A in Example~\ref{example_fBT} because one has to use the sums of line bundles to construct flag bundles. One can see that the iterated flag bundle admits an action of a torus of the base space if one takes equivariant bundles.  On the other hand, we consider much bigger torus action which comes from not only that of base space but also that of fibers. 
\end{remark}

%\begin{remark}\label{remark_exist_of_simply_connected_covering_K}
%\blue{(Now we do not need this remark. Shall we remove it?)}
%	According to the structure theorem, for every compact connected Lie group  $K$
%	there is a finite covering map $p\colon \widetilde{K}=K'\times T'\to K$ where 
%	$K'$ is  a  simply-connected compact Lie group and $T'$ is a torus. %\blue{Kuroki: I change ``simply connected'' into ``simply-connected'' in this paper as far as I noticed. Actually I don't care about them. But I think we should use one word in the research article because some of the referees don't like a lack of unity.}
%	Let $T$ be a maximal torus of $K$ containing $p(T')$, $S$ a  circle subgroup of $T$, and
%	$Z$  the centralizer of $S$. 
%	Let $\widetilde{S}$ be the identity component of $p^{-1}(S)$, then $\widetilde{S}$ is a circle subgroup of $\widetilde{K}$.
%	Let $\widetilde{Z}:=p^{-1}(Z)$. Then from the covering space argument, we can see that
%	$\widetilde{Z}$ is equal to the centralizer of $\widetilde{S}$. 
%	Let $p_1\colon \widetilde{K}\to K'$ be the projection, and 
%	let $Z':=p_1(\widetilde{Z})$.
%	Then it is easy to see that $K/Z$ is diffeomorphic to $K'/Z'$.
%	Therefore, when we consider the flag Bott tower associated to $(K_{\bullet}, Z_{\bullet})$ 
%	we may assume that all $K_j$ are  simply-connected up to flag Bott tower isomorphism.
%\end{remark}

We now give the second definition of  flag Bott tower of general Lie type  in the form of an orbit space similarly to the full flag Bott tower of type $A$ case (see~\cite[\S 2.2]{KLSS18}) as follows. 
\begin{definition}\label{def_fBT_quotient}
Let $(K_{\bullet}, Z_{\bullet}) = \{(K_j, Z_j) \mid 1 \leq j \leq m\}$.
	Given a family of homomorphisms $\{ \PP{\ell}{j} \colon Z_{j} \to T_{\ell} \mid 1 \leq j < \ell \leq m\}$,
	the space $F_m^{\varphi}$ is defined as the orbit space
	\[
	F_m^{\varphi} := (K_1 \times \cdots \times K_m)/(Z_1 \times \cdots \times Z_m),
	\]
	where $(z_{1},\ldots,z_{m})\in Z_1 \times \cdots \times Z_m$ acts on $(g_{1},\ldots,g_{m})\in K_1 \times \cdots \times K_m$ from the right by 
	\begin{equation}\label{eq:quotient}
		\begin{split}
			(g_1,\dots,g_m) &\cdot (z_1,\dots,z_m) \\
			&:= (g_1z_1, \PP{2}{1}(z_1)^{-1}g_2z_2, \prod_{j=1}^2\PP{3}{j}(z_j)^{-1} g_3z_3,\dots, \\
			&\quad \qquad \prod_{j=1}^{m-1}\PP{m}{j}(z_j)^{-1} g_mz_m).
			\end{split}
			\end{equation}
\end{definition}
This action is easily seen to be free, and hence, $F_m^{\varphi}$ is a smooth manifold.
Moreover, the space $F_m^{\varphi}$ has the structure of $K_m/Z_m$-fiber bundle over $F_{m-1}^{\varphi}$ whose classifying map $f_m$ is given by the composition
\[
\begin{tikzcd}
F_{m-1}^{\varphi} \arrow[r, ""] & B\left(\prod_{j=1}^{m-1} Z_j\right) \arrow[rr, "B (\prod_{j=1}^{m-1} \PP{m}{j})"] && BT_m \arrow[r, "B \iota"] & BK_m.
\end{tikzcd}
\]
Here the first map is the classifying map of the principal $\prod_{j=1}^{m-1} Z_j$-bundle
\[
\prod_{j=1}^{m-1} Z_j \hookrightarrow \prod_{j=1}^{m-1} K_j \to F_{m-1}^{\varphi},
\]
and $\iota:T_{m}\to K_{m}$ is the inclusion.
Therefore, $F_{\bullet}^{\varphi} := \{F_j^{\varphi} \mid 0 \leq j \leq m \}$ is an $m$-stage flag Bott tower of general Lie type associated to $(K_{\bullet}, Z_{\bullet})$.

Conversely, when $K_j$ are simply-connected for all $1\le j\le m$, 
we claim that every flag Bott tower of general Lie type associated to $(K_{\bullet}, Z_{\bullet})$ can be described as the orbit space as in Definition~\ref{def_fBT_quotient}, which implies that two definitions of
flag Bott tower are equivalent.
\begin{proposition}\label{prop_equiv_def_of_fBT}
	Let $F_{\bullet} = \{F_j \mid 0 \leq j \leq m \}$ be an $m$-stage flag Bott tower of general Lie type associated to $(K_{\bullet}, Z_{\bullet})$, where $K_j$ are simply-connected for all $1\le j\le m$. Then there exists a family of homomorphisms $\{ \PP{\ell}{j} \colon Z_{j} \to T_{\ell} \mid 1 \leq j < \ell \leq m \}$ such that $F_{\bullet}$ and $F_{\bullet}^{\varphi}$ are isomorphic as flag Bott towers. 
\end{proposition}
\begin{proof}
We show by induction.
	Assume that a flag bundle $F$ over $F_{m-1}^{\varphi}$ is associated to the classifying map
	\[
	f_m \colon F_{m-1}^{\varphi} \to BK_m
	\]
	with fiber $K_m/Z_m$, and also assume that $f_m$ factors through $BT_m$, i.e.
	there exist $\overline{f_m} \colon F_{m-1}^{\varphi} \to BT_m$ such that the following diagram commutes.
	\[
	\begin{tikzcd}[column sep = 0mm]
	F_{m-1}^{\varphi} \arrow[rr, "f_m"] \arrow[dr, "\overline{f_m}"'] && BK_m \\
	& BT_m \arrow[ru, "B \iota"']
	\end{tikzcd}
	\]	
	
	From the construction of $F_{m-1}^{\varphi}$, we have a principal $\prod_{j=1}^{m-1} Z_j$ bundle.
	Denote its classifying map by $u$ so that we have the pull-back of the universal $\prod_{j=1}^{m-1} Z_{j}$ bundle as follows:
	\[
	\begin{tikzcd}
	\prod_{j=1}^{m-1} Z_j \arrow[r,equal] \arrow[d,hook] & \prod_{j=1}^{m-1} Z_j \arrow[d, hook]\\
	\prod_{j=1}^{m-1} K_j \dar \rar & E  \dar \\
	F_{m-1}^{\varphi} \arrow[r, "u"] & B(\prod_{j=1}^{m-1} Z_j),
	\end{tikzcd}
	\]
	where $E = E(\prod_{j=1}^{m-1} K_j)$ can also serve as the universal space for any subgroup of $\prod_{j=1}^m K_j$.

	Since $E\to B(\prod_{j=1}^{m-1} Z_j)$ is a fibration, the bottom pull-back square is a homotopy pull-back at the same time.
	Thus, the bottom square can be understood as restricting $u$ on a contractible space $E$
	so that the homotopy fibre of $u$ is $\prod_{j=1}^{m-1}K_j$.
	That is, we have the following homotopy fibration
	\begin{equation}\label{eq_fibration_K_i_2}
	\prod_{j=1}^{m-1} K_j~\longrightarrow~ F_{m-1}^{\varphi} \stackrel{u}{\longrightarrow} B\left(\prod_{j=1}^{m-1} Z_j \right).
	\end{equation}
	
Note that each $K_j$ is compact and simply-connected, so that homotopy groups $\pi_1(K_j)$ and $\pi_2(K_j)$ are trivial for $1 \leq j \leq m$ (see~\cite[Proposition 7.5 in Chapter V]{BrockertomDieck85representations}). Hence $H^1(K_j) = H^2(K_j) = 0$ by the Hurewicz isomorphism theorem. 
	Since $Z_j$ is connected, $BZ_j$ is simply-connected for all $1 \leq j \leq m$. 
Recall that $F_{m-1}^{\varphi}$ admits the structure of a CW complex with even dimensional cells only.
	Combining these facts with the Serre spectral sequence with respect to the fibration~\eqref{eq_fibration_K_i_2} we see that
		\[
		u^* \colon H^2\left( B \left( \prod_{j=1}^{m-1} Z_j \right) ; \Z\right) \to H^2(F_{m-1}^{\varphi};\Z)
		\]
	is an isomorphism.
	
	Using the identification $[X, BS^1] \cong H^2(X;\mathbb{Z})$ for a topological space $X$ which has the homotopy type of a CW complex (see~\cite[Proposition 3.10]{hatcher2003vector}),
	we see that $\overline{f_m}$ factors up to homotopy as follows:
	\begin{equation}\label{eq_def_psi}
	\begin{tikzcd}[column sep = 0mm]
	F_{m-1}^{\varphi} \arrow[rr, "\overline{f_m}"] \arrow[rd, "u"'] && BT_m. \\
	& B(\prod_{j=1}^{m-1} Z_j) \arrow[ru,dashed,"\psi"']
	\end{tikzcd}
	\end{equation}
Furthermore, we obtain $\varphi \in \text{Hom}\left(\prod_{j=1}^{m-1}  Z_j,T_m\right)$
	such that $B\varphi=\psi$ through the following bijection:
		\[
	\begin{split}
	\textup{Hom}(Z_j,S^1) &\cong \textup{Hom}(\pi_1(Z_j),\Z) \quad
	\text{(see, for example,~\cite[Proposition 9.4]{BwyerWilkerson_compact_Lie_group98})} \\
	 &\cong \textup{Hom}(\pi_2(BZ_j), \Z) \\
	&\cong \textup{Hom}(H_2(BZ_j); \Z) \quad (\text{by the Hurewicz isomorphism theorem}) \\
	&\cong H^2(BZ_j;\Z) \quad (\text{by the universal coefficient theorem})
	\end{split}
	\]	
	Consider the isomorphism 
	\begin{equation}\label{eq_isom_hom}
	\text{Hom}\left(\prod_{j=1}^{m-1}  Z_j,T_m\right) \stackrel{\cong}{\longrightarrow}
	\left(\prod_{j=1}^{m-1}\text{Hom}(Z_j,T_m) \right) ,
	\end{equation}
	and denote the image of $\varphi$ under the map~\eqref{eq_isom_hom} by 
	\[
	\PP{m}{j} \colon Z_j \to T_m, \quad \text{ for } 1 \leq j \leq m-1.
	\]
	By Definition \ref{def_fBT_quotient}, 
	this sequence $(\PP{m}{j})_{1 \leq j \leq m-1}$ defines a bundle $F_{m}^{\varphi}$ over $F_{m-1}^{\varphi}$,
	which we will show is isomorphic to $F$.
	In fact, $F$ has the classifying map
	\[
	f_m \colon F_{m-1}^{\varphi} \stackrel{\overline{f_m}}{\longrightarrow} BT_m \stackrel{B \iota}{\longrightarrow} BK_m,
	\]
	and $F_{m}^{\varphi}$ has the classifying map
	\[
	\begin{tikzcd}
	F_{m-1}^{\varphi} \arrow[r, "u"] & B\left(\prod_{j=1}^{m-1} Z_j\right) \arrow[r, "B \varphi"] & BT_m \arrow[r, "B \iota"] & BK_m.
	\end{tikzcd}
	\]
	Since $\overline{f_m}$ is homotopic to $B \varphi\circ u$ by \eqref{eq_def_psi},
	they define the same bundle (see~\cite[Theorem 2.1]{cohen1998topology}). 
\end{proof}

\begin{remark}\label{remark_fBT_matrix}
Assume that all $Z_j$'s are maximal tori. If we fix the isomorphisms $Z_j\simeq (S^1)^{n_j+1}$, then
		$\PP{\ell}{j}:Z_{j}\to T_{\ell}$ can be represented by
	$(n_{\ell}+1) \times (n_{j}+1)$ matrices $A^{(\ell)}_{j}\in M_{n_{\ell}+1,n_{j}+1}(\Z)$ with integer entries
	through the obvious identification.
	For example, for $(K_j,Z_j)= (U(n_j+1),T_j)$ we have
\[
\begin{split}
&\PP{\ell}{j}(\mathrm{diag}(e^{\sqrt{-1}t_1},\ldots,e^{\sqrt{-1}t_{n_{j}+1}})) \\
& \qquad = 
\mathrm{diag}(e^{\sqrt{-1} \sum_{h=1}^{n_{j}+1} A^{(\ell)}_{j}(1,h) t_{h}},
\ldots,e^{\sqrt{-1} \sum_{h=1}^{n_{j}+1} A^{(\ell)}_{j}(n_{\ell}+1,{h}) t_{h}}),
\end{split}
\]
	as was introduced in \cite{KLSS18},
	where $A^{(\ell)}_{j}(k,h)$ is the $(k,h)$-entry of $A^{(\ell)}_{j}$ for $1\le k\le n_{\ell}+1$ and $1\le h\le n_{j}+1$.
\end{remark}

\begin{remark}\label{remark_fBT_and_Ajl}
	The simply-connectedness assumption on $K_j$ in Proposition~\ref{prop_equiv_def_of_fBT}
	can be weakened. The assumption is used only to assert $u^*$ is an isomorphism.
	When $u^*$ is surjective, we can choose a map $\psi$ which makes \eqref{eq_def_psi} commutative,
	and hence, homomorphisms $\PP{\ell}{j} \colon Z_{j} \to T_{\ell}$ for $1 \leq j < \ell \leq m$ 
	so that $F^{\varphi}_{\bullet}$ is isomorphic to $F_{\bullet}$ as flag Bott towers.
	In fact, in \cite[Proposition 2.11]{KLSS18}
	the case of full flag Bott manifolds of type~$A$ when $(K_j,Z_j)= (U(n_j+1),T_j)$ is considered
	and a particular choice for $\psi$ is made.
\end{remark}

We now have two descriptions which are equivalent when all $K_j$ are simply-connected of a flag Bott tower of general Lie type.
The latter description encodes not only the iterated bundle structure but
also an action of a torus as we will see in the next section.

%%%%%%%
\section{Equivariant cohomology rings of  flag Bott manifolds of \\ general Lie type}
\label{sec_equi_cohomology_ring_flag_Bott_manifolds}

For a topological space $X$ with an action of a topological group $G$, its equivariant cohomology $H^\ast_G (X)$ is defined to be
the singular cohomology  $H^\ast(X_G)$ of the Borel construction $X_G=EG\times_G X$ of $X$. Here $EG$ is a contractible space on
which $G$ acts freely, and $EG\times_G X=EG\times X/\sim$ where $(h, x)\sim(gh, gx)$ for any $(h,x)\in EG\times X$ and $g\in G$.
In this section, 
we define an action of a torus $\mathbb{T}$ on the flag Bott manifold $F_m^{\varphi}$ of general Lie type 
and compute the equivariant cohomology $H^\ast_{\mathbb{T}}(F_m^{\varphi})$.

For $1 \leq j \leq m$, let $K_j$ be a compact connected Lie group, $T_j \subset K_j$ be a maximal torus, 
and $Z_j \subset K_j$ be the centralizer of a circle subgroup of $T_j$. Consider an $m$-stage flag Bott tower  $\{F_{j}^{\varphi} \mid 0 \leq j \leq m\}$
determined by a family of homomorphisms $\{\PP{\ell}{j} \colon Z_{j} \to T_{\ell} \mid 1 \leq j < \ell \leq m\}$. 
We define a torus action on $F_{m}^{\varphi}$ as follows:
\begin{definition}\label{def:canonical_action_of_T}
	Let $\mathbb{T} = \prod_{j=1}^m T_j$.
	We have a well-defined action of $\mathbb{T}$ on $F_m^{\varphi}$ given by
	\[	
	(t_1,\dots,t_m) \cdot [g_1,\dots,g_m] = [t_1g_1,\dots,t_mg_m],
	\]
	where $(t_1,\dots,t_m) \in \mathbb{T}$ and $[g_1,\dots,g_m] \in F_m^{\varphi}$. 
	The well-definedness can be seen from the fact that the images of $\varphi_{j}^{(\ell)}$ lie in 
	the commutative groups $T_{\ell}$.
\end{definition}

Let us compute the equivariant cohomology ring of $F_m^{\varphi}$ with respect to the action of $\mathbb{T}$.
Let $\prod_{j=1}^{m-1}T_j$ act on $\prod_{j=1}^{m-1} K_j$ by the left multiplication.
The action of $\prod_{j=1}^{m-1}T_j$ on $\prod_{j=1}^{m-1} K_j$ commutes with that of $\prod_{j=1}^{m-1} Z_j$
given by $\{\PP{\ell}{j}\}$ as in~\eqref{eq:quotient},
so we have the principal $\prod_{j=1}^{m-1} Z_j$-bundle
\[
\prod_{j=1}^{m-1} Z_j \to E\times_{\prod_{j=1}^{m-1}T_j} \prod_{j=1}^{m-1} K_j \to E\times_{\prod_{j=1}^{m-1}T_j} F_{m-1}^{\varphi}.
\]
We denote its classifying map by
\[
\overline{u} \colon E\times_{\prod_{j=1}^{m-1}T_j} F_{m-1}^{\varphi} \to B\left(\prod_{j=1}^{m-1}Z_j\right).
\]

Since $F_m^{\varphi}$ can be identified with the associated flag bundle of $f_m \colon F_{m-1}^{\varphi} \to BK_m$ with fiber $K_m/Z_m$, its Borel construction with respect to the $\mathbb{T}$ action fits into the following pull-back diagram
\begin{equation}\label{eq_Borel_construction_fBT}
\begin{tikzcd}
K_m/Z_m \arrow[r, equal] \arrow[d,hook] & K_m/Z_m \arrow[d, hook]\\
E\mathbb{T}  \times_{\mathbb{T} } F_m^{\varphi} \rar \dar & BZ_m \dar \\
E\mathbb{T}  \times_{\mathbb{T} }F_{m-1}^{\varphi} \arrow[r, "\widetilde{f}_m"] & BK_m.
\end{tikzcd}
\end{equation}
where the factor $T_m$ in $\mathbb{T}$ acts trivially on $F_{m-1}^{\varphi}$.
Here, $\widetilde{f}_m \colon E\mathbb{T}  \times_{\mathbb{T} } F_{m-1}^{\varphi} \to BK_m$ is the following composition:
\begin{equation}\label{eq_def_of_tilde_f}
\begin{split}
&E\mathbb{T}  \times_{\mathbb{T} } F_{m-1}^{\varphi} \cong BT_m \times E\left(\prod_{j=1}^{m-1}T_j\right) \times_{\prod_{j=1}^{m-1} T_j} F_{m-1}^{\varphi} \\ 
& \qquad \stackrel{1 \times \overline{u}}{\xrightarrow{\hspace*{0.8cm}}}  B\left(T_m \times \prod_{j=1}^{m-1} Z_j\right) 
\stackrel{{B(\mathrm{mul})\circ B(1 \times \varphi)}}{\xrightarrow{\hspace*{2.6cm}}}
BT_m \stackrel{B \iota}{\longrightarrow} BK_m,
\end{split}
\end{equation}
where $\mathrm{mul}$ is the multiplication of $T_m$ and $\varphi := \prod_{j=1}^{m-1} \PP{m}{j}$. 

By applying Proposition~\ref{prop_cohomology_ring_of_associated_flag_bundle} to~\eqref{eq_Borel_construction_fBT} we have the following result:
\begin{theorem}\label{thm_equiv_cohomology_ring_of_FBM}
	 Let $K_j$ be a compact connected Lie group, $T_j \subset K_j$ a maximal torus with dimension $n_{j}$, and $Z_j \subset K_j$ the centralizer of a circle subgroup of $T_j$ for $1\leq j \leq m$.
	Let $\{F_{j}^{\varphi} \mid 0 \leq j \leq m\}$ be an $m$-stage flag Bott tower determined by a family of homomorphisms $\{\PP{\ell}{j} \colon Z_{j} \to T_{\ell} \mid 1 \leq j < \ell \leq m\}$. Let $R$ be a PID in which torsion primes of all $K_j$ are invertible. Then the equivariant cohomology ring of $F_m^{\varphi}$ with respect to the  
action of $\mathbb{T}$ defined by Definition~\ref{def:canonical_action_of_T} is 
	\[
	H^{\ast}_{\mathbb{T}}(F_m^{\varphi};R) \simeq H^{\ast}_{\mathbb{T}}(F_{m-1}^{\varphi};R) \otimes_{H^{\ast}(BK_m;R)} H^{\ast}(BZ_m;R)
	\]
	where the $H^{\ast}(BK_m;R)$-module structure on $H^{\ast}_{\mathbb{T}}(F_{m-1}^{\varphi};R)$ is induced  by $\widetilde{f}_m$. 
\end{theorem}

We use the following well-known identification to get an explicit formula for $H^{\ast}_{\mathbb{T}}(F_m^{\varphi};R)$:
\begin{equation}\label{eq_notation_u_y}
\begin{split}	
& H^{\ast}(BT_m;R) \simeq R[u_{m,k} \mid 1\leq k \leq n_m] \stackrel{\text{let}}{=}R[\mathbf{u}_m], \\
& H^{\ast}(BK_m;R) \simeq R[y_{m,k} \mid 1\leq k \leq n_m]^{W(K_m)} \stackrel{\text{let}}{=} R[\mathbf{y}_m]^{W(K_m)}, \\
& H^{\ast}(BZ_m;R) \simeq R[y_{m,k} \mid 1\leq k \leq n_m]^{W(Z_m)} \stackrel{\text{let}}{=} R[\mathbf{y}_m]^{W(Z_m)},
\end{split}
\end{equation}
where $u_{m,k}$ and $y_{m,k}$ are all degree $2$ elements for all $m,k$.
Here, we use the fact 
that the Weyl group $W(Z_{m})$ of $Z_{m}$ may be regarded as the subgroup of the Weyl group $W(K_{m})$ of $K_{m}$; therefore, the invariant polynomial ring    
$H^{\ast}(BK_{m};R) \simeq H^{\ast}(BT_{m};R)^{W(K_{m})}$ is a subring of the invariant polynomial ring $H^{\ast}(BZ_{m};R) \simeq H^{\ast}(BT_{m};R)^{W(Z_{m})}$ (see~\cite{Borel53Sur}). 
We also note that we use the symbol $\mathbf{u}_m$ as the generators defined from the acting torus $T_{m}$ and the symbol  
$\mathbf{y}_m$ as the generators defined from the maximal torus $T_{m}\subset Z_{m}\subset K_{m}$.

Since $T_m$ acts trivially on $F_{m-1}^{\varphi}$, we have that
\[
\begin{split}
H^{\ast}_{\mathbb T}(F_{m}^{\varphi};R) &\simeq H^{\ast}_{\mathbb T}(F_{m-1}^{\varphi};R) \otimes_{H^{\ast}(BK_m;R)} H^{\ast}(BZ_m;R) \\
&\simeq H^{\ast}_{\mathbb T}(F_{m-1}^{\varphi};R) \otimes_{R[\mathbf{y}_m]^{W(K_m)}} R[\mathbf{y}_m]^{W(Z_m)} \\
&\simeq \left( H^{\ast}_{{\mathbb T}'}(F_{m-1}^{\varphi};R) \otimes H^{\ast}(BT_m) \right) \otimes_{R[\mathbf{y}_m]^{W(K_m)}} R[\mathbf{y}_m]^{W(Z_m)} \\
&\simeq \left( H^{\ast}_{{\mathbb T}'}(F_{m-1}^{\varphi};R) \otimes R[\mathbf{u}_m] \right) \otimes_{R[\mathbf{y}_m]^{W(K_m)}} R[\mathbf{y}_m]^{W(Z_m)} 
\end{split}
\]
where ${\mathbb T}' = \prod_{j=1}^{m-1} T_j$.
%\blue{Kuroki: Please check the following expression.} 
It is easy to check that the image of $y_{m,k}\in H^{\ast}(BT_{m})$ $(k=1,\ldots, n_{m})$ by the induced homomorphism of the composition maps~\eqref{eq_def_of_tilde_f} can be written by 
\[
u_{m,k} + \sum_{j =1}^{m-1}(\PP{m}{j})^{\ast}(y_{m,k}),
\]
where $(\PP{m}{j})^{\ast}:H^{\ast}(BT_{m})\to H^{\ast}_{{\mathbb T}'}(F_{m-1}^{\varphi})$ is the induced homomorphism of $B\PP{m}{j}\circ \overline{u}$.
Using the inductive application of the above argument, we have the following explicit formula:
\begin{corollary}\label{cor_equiva_cohomo_ring_of_fBT_general_type}
	Let $\mathbf{u}_j$ and $\mathbf{y}_j$ stand for $(u_{j,1},\dots,u_{j,n_j})$ and $(y_{j,1},\dots,y_{j,n_j})$, respectively for $j=1,\ldots,m$.  Then, we have
	\[
	H^{\ast}_{\mathbb{T}}(F_{m}^{\varphi};R) \simeq R[\mathbf{u}_1,\dots,\mathbf{u}_m] \otimes_R (R[\mathbf{y}_1,\dots,\mathbf{y}_m])^{W(Z)}\big/ \langle I_1,\dots, I_m \rangle
	\]
	where $W(Z) = \prod_{j=1}^m W(Z_j)$, and $I_{\ell}$, for $1\leq \ell \leq m$, is the ideal generated by the polynomials
	\[
	h(\mathbf{y}_{\ell}) - h\left(\mathbf{u}_{\ell} + \sum_{j =1}^{\ell-1}\Phi_{j}^{(\ell)}(\mathbf{y}_{\ell})\right) 
	\]
	for $h \in R[\mathbf{y}_{\ell}]^{W(K_{\ell})}$ and
\[
\Phi_{j}^{(\ell)}(\mathbf{y}_{\ell}):=\left((\PP{\ell}{j})^{\ast}(y_{\ell,1}),\dots,(\PP{\ell}{j})^{\ast}(y_{\ell,n_\ell})\right).
\]
\end{corollary}

\begin{corollary}\label{cor_equiv_cohomology_type_A}
Suppose that $F_{\bullet} = \{F_j \mid 0 \leq j \leq m \}$ is an $m$-stage flag Bott manifold of type A defined by
a set of integer matrices 
$\{\A{\ell}{j} \in M_{(n_{\ell}+1) \times (n_{j}+1)}(\Z) \mid 1 \leq j < \ell \leq m\}$ as in~Remark~\ref{remark_fBT_matrix}.
Then, we have that
	\[
	H^*_{\mathbb{T}}(F_m ; \Z)
	\cong \Z[\mathbf{u}_{j}, \mathbf{y}_j \mid 1 \leq j \leq m]/\langle I_1',\dots,I_m' \rangle,
	\]
	where $I_{\ell}'$ is the ideal generated by the polynomials
	\[
	c_k(\mathbf{y}_{\ell}) - c_k\left(\mathbf{u}_{\ell} + \sum_{j = 1}^{\ell-1} \Phi_{j}^{(\ell)} (\mathbf{y}_{\ell}) \right).
	\]
	Here $c_k$ are the $k$th elementary symmetric polynomials for $1 \leq k \leq n_{\ell}+1$.
	Since the homomorphism $\PP{\ell}{j}:T_{j}\to T_{\ell}$ is determined by the matrix $\A{\ell}{j}:H_{2}(BT_{j})=\mathbb{Z}^{n_{j}+1}\to H_{2}(BT_{\ell})= \mathbb{Z}^{n_{\ell}+1}$, by letting
$\A{\ell}{j}=(\A{\ell}{j}(k,h))\in M_{(n_{\ell}+1)\times (n_{j}+1)}(\mathbb{Z})$ for $1\le k\le n_{\ell}+1$ and $1\le h\le n_{j}+1$,
the $k$th entry of $\sum_{j = 1}^{\ell-1} \Phi_{j}^{(\ell)} (\mathbf y_{\ell})$, $1\le k\le n_{\ell}+1$, is 
\[
\sum_{j = 1 }^{\ell-1} \left(\sum_{h=1}^{n_{j}+1}\A{\ell}{j}(k,h)\cdot y_{j,h}.\right)
\]
\end{corollary}

The next two remarks show that \eqref{def:canonical_action_of_T} 
may define different torus actions on isomorphic $F^{\varphi}_{\bullet}$'s depending on the defining data 
$(K_\bullet,Z_\bullet)$ and $\{\PP{\ell}{j}\}$.

\begin{remark}\label{rmk_equivariant_cohomology_typeA}
Let $(K_1,Z_1)=(SU(2),T_{SU(2)})$. The corresponding 
flag Bott manifold of general Lie type is $SU(2)/T_{SU(2)}\simeq \C P^1$.
Then $S^1\simeq T_{SU(2)}$ acts on $SU(2)/T_{SU(2)}$ by the left multiplication (see~Definition~\ref{def:canonical_action_of_T}).
This action is not effective but $\mathrm{diag}(-1,-1) \in T_{SU(2)}$ acts trivially.
By Corollary~\ref{cor_equiva_cohomo_ring_of_fBT_general_type}, the equivariant cohomology of $SU(2)/T_{SU(2)}$ with this action is
\[
H^*_{T_{SU(2)}}(SU(2)/T_{SU(2)};\Z) \simeq \Z[u_{1,1},y_{1,1}]/I, 
\]
where $I$ is the ideal generated by 
\[
u_{1,1}^{2} -y_{1,1}^{2}
\]
because $W(SU(2))\simeq \mathbb{Z}_{2}$ acts on $H^{*}(BS^{1})\simeq \mathbb{Z}[x]$ by $x\mapsto -x$.
Hence,  
we have that
\begin{equation}\label{eq_equivariant_cohomology_SU}
H^*_{T_{SU(2)}}(SU(2)/T_{SU(2)};\Z) \simeq \Z[u_{1,1}, y_{1,1}]/\langle u_{1,1}^2 - y_{1,1}^2 \rangle.
\end{equation}

On the other hand, 
for $(K_1,Z_1)=(SO(3),T_{SO(3)})$ the corresponding 
flag Bott manifold of general Lie type is again $SO(3)/T_{SO(3)} \simeq \C P^1$.
This time, $S^1\simeq T_{SO(3)}$ acts effectively
on $SO(3)/T_{SO(3)}$ by the left multiplication.
Theorem~\ref{thm_equiv_cohomology_ring_of_FBM} is not applicable for $R=\Z$ coefficients
since two is the torsion prime of $SO(3)$. However, the standard argument shows
\begin{equation}\label{eq_equivariant_cohomology_CP1}
H^*_{T_{SO(3)}}(SO(3)/T_{SO(3)}; \Z) \simeq \Z[u,v]/\langle uv \rangle.
\end{equation}
One can see that rings~\eqref{eq_equivariant_cohomology_SU} and~\eqref{eq_equivariant_cohomology_CP1} are not isomorphic, and this is because they represent the equivariant cohomology rings with different $S^1$-actions.
\end{remark}

\begin{remark}
Consider $(K_{\bullet}, Z_{\bullet})=\{(SU(2), T_{1}), (SU(2),T_{2})\}$.
Since $SU(2)/T\simeq \C P^1$ and $\pi_{2}(BSU(2))$ is trivial, $F_{2}^{\varphi} \simeq \C P^1\times \C P^1$
for any $\PP{1}{2}: T_1 \to T_2$.
On the other hand, 
we can see the torus actions defined by \eqref{def:canonical_action_of_T} are distinct for different $\PP{1}{2}$.
In this sense, Definition~\ref{def_fBT_quotient} encodes more structures (torus equivariant structures)
	than Definition~\ref{def_fBT_pullback}.
\end{remark}

\begin{remark}
	For an $m$-stage flag Bott manifold $F_m$ associated to $\{(U(n_{j}+1), U(1)^{n_{j}+1})\ |\ j=1,\ldots, m\}$ (see Example~\ref{example_fBT}), the torus $\mathbb{T}$ does not act effectively on $F_m$. If we write $t_j = \diag(t_{j,1},\dots,t_{j,n_j+1}) \in U(1)^{n_{j}+1}=T_j \subset U(n_j+1)$, the subtorus
	\[
	\mathbf{T} := \{(t_1,\dots,t_m) \in \mathbb{T} \mid t_{1,n_1+1} = \cdots = t_{m,n_m+1} = 1 \} \cong (S^1)^{n_1+ \cdots + n_m}
	\]
	acts effectively on $F_m$ (see~\cite[\S 3.1]{KLSS18}). Then the equivariant cohomology ring $H^{\ast}_{\mathbf{T}}(F_m; \Z)$ with respect to the effective torus action of $\mathbf{T}$ is given by 
	\[
	H^{\ast}_{\mathbf{T}}(F_m; \Z) \cong \Z[\mathbf{u}_{j}, \mathbf{y}_j \mid 1 \leq j \leq m ]/\langle u_{1,n_1+1},\dots,u_{m,n_m+1}, I_1',\dots,I_m' \rangle,
	\]
	where $I_j'$ are defined in Corollary~\ref{cor_equiv_cohomology_type_A}.
	Moreover, by ignoring the generators $\mathbf{u}_j$ we get the singular cohomology ring of $F_m$:
	\[
	H^{\ast}(F_m;\Z) \cong \Z[\mathbf{y}_j \mid 1 \leq j \leq m] / \langle J_1,\dots,J_m \rangle,
	\]
	where $J_{\ell}$ is the ideal generated by the polynomials
	\[
	c_k(\mathbf{y}_{\ell}) - c_k \left( \sum_{j=1}^{\ell-1} \Phi_{j}^{(\ell)} (\mathbf{y}_{\ell}) \right).
	\]
\end{remark}

\begin{example}
\label{typeC}
	Consider $F_3 := Sp(3) \times_{T^3} Sp(3) \times_{T^3} Sp(2)/T^2$ with $\PP{\ell}{j} \colon T^{n_{j}} \to T^{n_{\ell}}$
	$((n_1,n_2,n_3)=(3,3,2))$ for $1 \leq j < \ell \leq 3$ determined by the matrices~$\A{\ell}{j}$:
\begin{align*}
	&\A{2}{1} = \begin{bmatrix}
	\A{2}{1}(1,1) & \A{2}{1}(1,2) & \A{2}{1}(1,3) \\ 
	\A{2}{1}(2,1) & \A{2}{1}(2,2) & \A{2}{1}(2,3) \\ 
	\A{2}{1}(3,1) & \A{2}{1}(3,2) & \A{2}{1}(3,3) 
	\end{bmatrix} \in M_{3 \times 3}(\Z), \\
	&\A{3}{1} = \begin{bmatrix}
	\A{3}{1}(1,1) & \A{3}{1}(1,2) & \A{3}{1}(1,3) \\ 
	\A{3}{1}(2,1) & \A{3}{1}(2,2) & \A{3}{1}(2,3) \\ 
	\end{bmatrix} \in M_{2 \times 3}(\Z), \\ 
	&\A{3}{2} = \begin{bmatrix}
	\A{3}{2}(1,1) & \A{3}{2}(1,2) & \A{3}{2}(1,3) \\ 
	\A{3}{2}(2,1) & \A{3}{2}(2,2) & \A{3}{2}(2,3) \\ 
	\end{bmatrix} \in M_{2 \times 3}(\Z). 
	\end{align*}
%	\[
%	\A{2}{1} = \begin{bmatrix}
%	\bolda{2}{1}{1} \\ \bolda{2}{2}{1} \\ \bolda{2}{3}{1}
%	\end{bmatrix} \in M_{3 \times 3}(\Z), \quad
%	\A{3}{1} = \begin{bmatrix}
%	\bolda{3}{1}{1} \\ \bolda{3}{2}{1}
%	\end{bmatrix} \in M_{2 \times 3}(\Z), \quad 
%	\A{3}{2} = \begin{bmatrix}
%	\bolda{3}{1}{2} \\ \bolda{3}{2}{2}
%	\end{bmatrix} \in M_{2 \times 3}(\Z). 
%	\]
	Here, we think of $Sp(n)$ as matrices with quaternion entries, and its maximal torus is chosen to be the diagonal matrices  whose entries are complex numbers with unit lengths.
	Then by Corollary~\ref{cor_equiva_cohomo_ring_of_fBT_general_type} (or by \cite[\S 6.2]{FP98Schubert}), we have that
	\[
	H^*_{\mathbb{T}}(F_3 ; \Z) \cong \mathbb{Z}[\mathbf u_1, \mathbf u_2, \mathbf u_3, \mathbf y_1, \mathbf y_2, \mathbf y_3]/ \langle I_1, I_2, I_3 \rangle,
	\]
	where $\mathbf y_j = (y_{j,1},\dots,y_{j,n_j})$, $\mathbf u_j = (u_{j,1},\dots,u_{j,n_j})$,
	$y_{j,k}$ and $u_{j,k}$ are degree~$2$ elements, and
	\begin{align*}
	I_1 &= (1 + y_{1,1}^2)(1+y_{1,2}^2)(1+y_{1,3}^2) - (1+u_{1,1}^2)(1+u_{1,2}^2)(1+u_{1,3}^2), \\
	I_2 &= (1+y_{2,1}^2)(1+y_{2,2}^2)(1+y_{2,3}^2)-\prod_{k=1}^{3}\left\{1+\left(u_{2,k} + \left(\sum_{h=1}^{3} \A{2}{1}(k,h)\cdot y_{1,h} \right) \right)^2\right\}  \\
%	& \qquad - (1+(u_{2,1} + \red{\sum_{i=1}^{3} \A{2}{1}(1,i) \cdot y_{2,1}})^2)(1+(u_{2,2} + \red{\sum_{i=1}^{3} \A{2}{1}(2,i) \cdot y_{2,2}})^2)(1+(u_{2,3} + \red{\sum_{i=1}^{3} \A{2}{1}(3,i) \cdot y_{2,3}})^2),\\
	I_3 &= (1+y_{3,1}^2)(1+y_{3,2}^2)-\prod_{k=1}^{2}\left\{1+\left(u_{3,k} + \sum_{j=1}^2 \left( \sum_{h=1}^{3} \A{3}{j}(k,h)\cdot y_{j,h} \right)  \right)^2\right\}.
%	I_2 &= (1+y_{2,1}^2)(1+y_{2,2}^2)(1+y_{2,3}^2) \\
%	& \qquad - (1+(u_{2,1} + \bolda{2}{1}{1} \cdot \mathbf y_1)^2)(1+(u_{2,2} + \bolda{2}{2}{1} \cdot \mathbf{y}_2)^2)(1+(u_{2,3} + \bolda{2}{3}{1} \cdot \mathbf{y}_3)^2),\\
%	I_3 &= (1+y_{3,1}^2)(1+y_{3,2}^2) \\
%	&\qquad -(1+(u_{3,1} + \bolda{3}{1}{1} \cdot \mathbf{y}_1 + \bolda{3}{1}{2} \cdot \mathbf{y}_2)^2)
%	(1+(u_{3,2} + \bolda{3}{2}{1} \cdot \mathbf{y}_1 + \bolda{3}{2}{2} \cdot \mathbf{y}_2)^2).
	\end{align*}
\end{example}

\begin{example}
	Consider $F_3 := SU(4) \times_{T^3} Sp(3) \times_{T^3} G_{2}/T^2$ with $\PP{\ell}{j} \colon T^{n_{j}} \to T^{n_{\ell}}$ 
	$((n_1,n_2,n_3)=(3,3,2))$ for $1 \leq j < \ell \leq 3$.
Since two is the torsion prime of $G_{2}$, 
we can apply our theorem with the coefficients in $R=\mathbb{F}_{3}$, i.e., the finite field of order $3$.
Let 
the homomorphisms $\{\PP{\ell}{j} \mid 1 \leq j < \ell \leq 3\}$ be determined by matrices 
$\A{\ell}{j} \in M_{n_{\ell} \times n_{j}}(\Z)$,
where we identified $T^2\subset G_2$ with the set of the matrix $\diag(t_{1},t_{2},t_{1}^{-1}t_{2}^{-1})\in SU(3)\subset G_{2}\subset SO(7)$ for $SU(3)=G_{2}\cap SO(6)$ following \cite{Ishig10}.
Then, by Corollary~\ref{cor_equiva_cohomo_ring_of_fBT_general_type} together with \cite[p.300]{Ishig10},
there are $\mathbf y_j = (y_{j,1},\dots,y_{j,n_j})$, $\mathbf u_j = (u_{j,1},\dots,u_{j,n_j})$ for $j=1,2,3$ such that 
	\[
	H^*_{\mathbb{T}}(F_3 ; \mathbb{F}_{3}) \cong \mathbb{F}_{3}[\mathbf u_1, \mathbf u_2, \mathbf u_3, \mathbf y_1, \mathbf y_2, \mathbf y_3]/ \langle L_1, L_2, L_3 \rangle.
	\]
where 
	\begin{align*}
	L_1 &= (1 + y_{1,1})(1+y_{1,2})(1+y_{1,3}) - (1+u_{1,1})(1+u_{1,2})(1+u_{1,3}), \\
	L_2 &= (1+y_{2,1}^2)(1+y_{2,2}^2)(1+y_{2,3}^2)-\prod_{k=1}^{3}\left\{1+\left(u_{2,k} + \left(\sum_{h=1}^{3} \A{2}{1}(k,h)\cdot y_{1,h} \right)\right)^2\right\} \\
%	& \qquad - (1+(u_{2,1} + \bolda{2}{1}{1} \cdot \mathbf y_1)^2)(1+(u_{2,2} + \bolda{2}{2}{1} \cdot \mathbf{y}_2)^2)(1+(u_{2,3} + \bolda{2}{3}{1} \cdot \mathbf{y}_3)^2),\\
	L_3 &= (1+(y_{3,1}-y_{3,2})^2)(1+y_{3,1}^{2}y_{3,2}^2(y_{3,1}+y_{3,2})^{2}) \\
	&\qquad -\left(1+h_{4}\left(\mathbf{u}_{3} + \Phi_{1}^{(3)}(\mathbf{y}_3) + \Phi_{2}^{(3)}(\mathbf{y}_3)\right)\right)
	\left(1+h_{12}\left(\mathbf{u}_{3} + \Phi_{1}^{(3)}(\mathbf{y}_3) + \Phi_{2}^{(3)}(\mathbf{y}_3)\right)\right).
%	&\qquad -(1+h_{4}(\mathbf{u}_{3} + A_{1}^{(3)} \cdot \mathbf{y}_1 + A_{2}^{(3)} \cdot \mathbf{y}_2))
%	(1+h_{12}(\mathbf{u}_{3} + A_{1}^{(3)} \cdot \mathbf{y}_1 + A_{2}^{(3)} \cdot \mathbf{y}_2)).
	\end{align*}
In the last relation, we use the notations  
\begin{align*}
h_{4}(x_{1},x_{2})&= (x_{1}-x_{2})^{2}, \\
h_{12}(x_{1},x_{2})&= x_{1}^{2}x_{2}^{2}(x_{1}+x_{2})^{2},
\end{align*}
for 
\begin{align*}
x_{1}&=u_{3,1}+  \left( \sum_{h=1}^{3} \A{3}{1}(1,h)\cdot y_{1,h} \right)+\left(\sum_{h=1}^{3} \A{3}{2}(1,h) \cdot y_{2,h} \right),\\
%\bolda{3}{1}{1} \cdot \mathbf y_1+\red{\bolda{3}{2}{1}} \cdot \mathbf y_2, \\
x_{2}&=u_{3,2}+  \left( \sum_{h=1}^{3} \A{3}{1}(2,h)\cdot y_{1,h} \right)+\left(\sum_{h=1}^{3} \A{3}{2}(2,h) \cdot y_{2,h} \right).
%\red{\bolda{3}{1}{2}} \cdot \mathbf y_1+\bolda{3}{2}{2} \cdot \mathbf y_2.
\end{align*}
Here, $\A{\ell}{j}(k,h)$ is the mod $3$ reduction of the $(k,h)$-entry of the matrix~$\A{\ell}{j}$.
\end{example}

%%%%%%%%%%%%%%%%%%%%%%%%%%%%%%%%%%%%%%%%%%%%%%%%%%%%%%%%%%%%%%%%%%%%%%%%
%\bibliographystyle{abbrv}
%\bibliography{ref}

\end{document}